\documentclass[letterpaper, 10pt, conference]{ieeeconf}    
\IEEEoverridecommandlockouts                            

\usepackage{amsmath,amssymb}
\usepackage{graphicx}
\usepackage{times}
\usepackage{bm} 
\usepackage{siunitx}
\usepackage{subfloat}
\usepackage{physics}
\usepackage{algorithmic}
\usepackage{algorithm}
\usepackage{cite}
\usepackage{caption}
\let\proof\relax
\let\endproof\relax
\usepackage{amsthm}
\theoremstyle{definition}

\newtheorem{assumption}{Assumption}
\newtheorem{theorem}{Theorem}

\newtheorem{remark}{Remark}
\newtheorem{lemma}{Lemma}

\newtheorem*{notation}{Notations}
\newtheorem{problem}{Problem}

\title{Acceleration of Moment Bound Optimization for Stochastic Chemical Reactions using Reaction-wise Sparsity of Moment Equations}
\author{
  Tomoki Sadatoshi, Antonis Papachristodoulou, Yutaka Hori
  \thanks{This work was supported in part by JSPS KAKENHI Grant Number JP24KK0267 and by UKRI-EPSRC via grants EP/Y014073/1, EP/X031470/1 and UKRI2108. For the purpose of Open Access, the authors have applied a CC BY public copyright licence to any Author Accepted Manuscript (AAM) version arising from this submission.}
  \thanks{T. Sadatoshi and Y. Hori are with the Department
of Applied Physics and Physico-Informatics, Keio University, Kanagawa 223-8522 Japan. {\tt \{tsadatoshi.29, yhori\}@keio.jp}}
\thanks{A. Papachristodoulou is with the Department of Engineering Science, Oxford University, Oxford OX1 3PJ United Kingdom. {\tt antonis@eng.ox.ac.uk}}
}
\date{February 2026}

\begin{document}

\maketitle
\begin{abstract}
Moment dynamics in stochastic chemical kinetics often involve an infinite chain of coupled equations, where lower-order moments depend on higher-order ones, making them analytically intractable. Moment bounding via semidefinite programming provides guaranteed upper and lower bounds on stationary moments. However, this formulation suffers from the rapidly growing size of semidefinite constraints due to the combinatorial growth of moments with the number of molecular species.
In this paper, we propose a sparsity-exploiting matrix decomposition method for semidefinite constraints in stationary moment bounding problems to reduce the computational cost of the resulting semidefinite programs. Specifically, we characterize the sparsity structure of moment equations, where each reaction involves only a subset of variables determined by its reactants, and exploit this structure to decompose the semidefinite constraints into smaller ones. We demonstrate that the resulting formulation reduces the computational cost of the optimization problem while providing practically useful bounds.
\end{abstract}

\section{Introduction}
Model-based approaches for the analysis and design of biomolecular systems have become increasingly important in synthetic biology applications\cite{kitano2002systems,del2015biomolecular}. 
A key challenge in the analysis of biomolecular reaction systems is their intrinsic stochasticity, which arises from random molecular interactions due to the low copy numbers of molecular species in a cell \cite{elowitz2002stochastic}.
Such stochastic dynamics are commonly described by the Chemical Master Equation (CME), which governs the time evolution of the probability distribution of molecular copy numbers \cite{gillespie1992rigorous}.
In most practical cases, the CME is an infinite-dimensional linear equation, making analytical solutions intractable. This has led to extensive research on approximate methods for analyzing stochastic dynamics \cite{elf2003fast, gillespie2000chemical}.
Among these, moment-based approaches enable an intuitive characterization of the copy number distributions using statistics such as mean and variance. 
These approaches analyze moment equations induced by the CME through approximations\cite{khammash2022cybergenetics}.
Moment closure \cite{naasell2003extension, hespanha2008moment, singh2010approximate} is a widely used approach that \textit{closes} the hierarchy of moment equations by approximating higher-order moments in terms of lower-order ones. However, a key limitation of such approximations is the lack of systematic error bounds, making it difficult to assess their accuracy.

To address this issue, semidefinite programming (SDP)-based methods have been proposed for bounding moments, enabling the computation of mathematically guaranteed upper and lower bounds despite the intractable nature of the underlying moment equations \cite{sakurai2017convex,sakurai2018optimization,sakurai2022interval,dowdy2018bounds,ghusinga2017exact,kuntz2019bounding,li2025moment,sakurai2018bounding,dowdy2018dynamic,holtorf2024tighter}.
These methods formulate optimization problems over convex feasible sets of moment variables defined by the moment equations and necessary semidefinite conditions that valid moment sequences must satisfy, and provide practically useful bounds for both stationary \cite{sakurai2017convex,sakurai2018optimization,sakurai2022interval,dowdy2018bounds,ghusinga2017exact,kuntz2019bounding} and transient moments \cite{sakurai2018bounding,dowdy2018dynamic,holtorf2024tighter} in a broad class of reaction systems.
Despite these advances, a major 
challenge lies in the computational cost. 
Specifically, the number of possible moment combinations grows combinatorially with the number of molecular species, leading to a rapid increase in the size of the matrix constraints in the optimization.
A known approach in semidefinite programming exploits sparsity in the underlying problem structure to 
reduce the size of the matrix constraints \cite{fukuda2001exploiting,wang2021chordal,zheng2021chordal,klep2025exploiting}. 
A similar idea may be applicable to the optimization problems over moments in stochastic chemical reaction systems. However, it remains unclear how sparsity manifests in the corresponding moment equations and how this structure can be effectively exploited.

In this paper, we propose a method to accelerate optimization computation in moment bounding optimization for stochastic chemical reactions by exploiting sparsity in the moment equations and decomposing the resulting semidefinite matrix constraints.
Specifically, our contributions are twofold:
(i) we characterize the sparsity structure of moment equations that emerges when they are decomposed in a \textit{reaction-wise} manner; and
(ii) we exploit this structure to decompose the semidefinite constraints into smaller ones, thereby accelerating the resulting semidefinite programs.
Despite introducing a relaxation, the proposed decomposition is closely related to chordal decomposition, a well-known technique in semidefinite programming that exploits sparsity to decompose large semidefinite constraints into smaller ones, and is therefore expected to provide tight bounds in practice\cite{fukuda2001exploiting}. 
We demonstrate that the proposed method reduces computation time while maintaining bounding accuracy in numerical experiments on a gene expression system with 7 molecular species and 14 reactions.

\begin{notation}
    $\mathbb{P}(\cdot)$ denotes a probability distribution, and $\mathbb{E}[\cdot]$ denotes the moments with respect to the distribution $\mathbb{P}(\cdot)$.
For vectors $\bm{x} = [x_1, x_2, \dots, x_n]^{\top}$ and $\bm{\alpha} = [\alpha_1, \alpha_2, \dots, \alpha_n]^{\top}\in \mathbb{N}_0^n$, a multi-index exponent is defined by $\bm{x}^{\bm{\alpha}} = x_1^{\alpha_1} x_2^{\alpha_2} \cdots x_n^{\alpha_n}$.
For a vector $\bm{\alpha}$, $|\bm{\alpha}| := \sum_{j=1}^{n} \alpha_j$.
$\bm{v}(\Phi)$ denotes the vector of the monomials in the set $\Phi$ arranged in graded lexicographic order. $\mathcal{M}_{\delta}$ denotes the set of all monomials $\bm{x}^{\bm{\alpha}}$ with degree less or equal to $\delta$, \textit{i.e.,} $\{\bm{x}^{\bm{\alpha}} \mid |\bm{\alpha}| \le \delta\}$.

\end{notation}

\section{Formulation and Problem Statement}
\subsection{Moment equation for stochastic chemical reaction}
We consider the stochastic kinetics of molecular copy numbers in chemical reaction systems. 
Suppose the reaction system consists of $n$ molecular species
$X_1, X_2, \cdots, X_n$ and $r$ chemical reactions
$\mathcal{R}_1, \mathcal{R}_2, \cdots, \mathcal{R}_r$.
Let $x_j$ denote the copy number of molecular species
$X_j$, and define the vector of copy numbers $\bm{x} := [x_1, x_2, \cdots, x_n]^\top \in \mathcal{X} \subseteq \mathbb{N}_0^n$, which characterizes the state of the reaction system. 
For each reaction $\mathcal{R}_i$, we denote a propensity function and  a stoichiometric vector by $w_i(\bm{x})$ 
and $\bm{s}_i = [s_{i1}, s_{i2}, \ldots, s_{in}]^\top \in \mathbb{Z}^n$, respectively, where $w_i(\bm{x})$ characterizes the reaction rate parameterized by the rate constant $\theta_i$, and $\bm{s}_i$ represents the change in molecular copy numbers by reaction $\mathcal{R}_i$.
Assuming mass-action kinetics, the propensity function $w_i(\bm{x})$ for each reaction $\mathcal{R}_i$ can be represented 
by either of the reaction types listed in Table~\ref{propent}. 

\begin{table}[b]
  \centering
  \caption{Elementary reactions and propensity functions}
  \label{propent}
  \begin{tabular}{ccc}
      \hline
      Name & Reaction & Propensity function \\ \hline\hline
      zero-th order & $\emptyset \rightarrow \text{Products}$ & $\theta_i$ \\
      unimolecular & $X_j \rightarrow \text{Products}$ & $\theta_ix_j$ \\
            homo-bimolecular & $2X_j \rightarrow \text{Products}$ & $\theta_ix_j(x_j-1)/{2}$ \\
      hetero-bimolecular & $X_j + X_k \rightarrow\text{Products}$ & $\theta_ix_jx_k$ \\
 \hline
  \end{tabular}
\end{table}

The stochastic evolution of the copy numbers $\bm{x}$ is modeled by a continuous-time  Markov process on the discrete state space $\mathcal{X}$. 
The time evolution of the probability distribution $\mathbb{P}(\bm{x})$, or the copy number distribution, is governed by the Chemical Master Equation (CME) \cite{gillespie1992rigorous}. 
Correspondingly, the 
dynamics of its  
moments $\mathbb{E}[\bm{x}^{\bm{\alpha}}]$ are given by the following moment equations\cite{khammash2022cybergenetics}. 
\begin{align}\label{momeq1}
  \begin{split}
    \frac{d}{dt} \mathbb{E}[\bm{x}^{\bm{\alpha}}]
    =\sum_{i=1}^{r}\mathbb{E}[\{(\bm{x}+\bm{s}_i)^{\bm{\alpha}}-\bm{x}^{\bm{\alpha}}\}w_i(\bm{x})].
  \end{split}
\end{align}
Since each propensity function $w_i(\bm{x})$ is a polynomial in $\bm{x}$ (see Table \ref{propent}), the moment equation becomes a linear ODE with respect to the moments. However, the right-hand side of the equation depends on the moments of order $|\bm{\alpha}|+1$ when there is a homo- or hetero-bimolecular reaction since $w_i(\bm{x})$ is a quadratic function. 
As a consequence, eq.~\eqref{momeq1} constitutes an infinite hierarchy of equations, making  exact analytic solutions intractable in practical reaction systems. 
The resulting infinite-dimensional system poses a fundamental challenge for the rigorous analysis of stochastic chemical systems. 

Throughout this paper, we make the following assumptions to guarantee the well-posedness of the moment equations.
\begin{assumption}
    (i) The CME admits a stationary distribution, (ii) the Markov chain is non-explosive, (iii) all moments $E[\bm{x}^{\bm{\alpha}}]$ converge to finite values.
\end{assumption}

\subsection{Semidefinite programming for bounding moments}

In what follows, we restrict our attention to \textit{stationary} moments and study the problem of computing their upper and lower bounds without resorting to approximation schemes such as moment closures \cite{naasell2003extension, hespanha2008moment, singh2010approximate}. 
The key idea is to use necessary conditions satisfied by stationary moments and to compute the bounds by optimizing over the set of moment vectors consistent with these conditions. It is known that this approach leads to a semidefinite program and can provide practically useful bounds for a broad class of reaction systems \cite{sakurai2017convex,sakurai2018optimization,sakurai2022interval,dowdy2018bounds,kuntz2019bounding,ghusinga2017exact,li2025moment}.

Specifically, stationary moments must satisfy the truncated stationary moment equations. 
Due to the linearity of the moment equation \eqref{momeq1}, the truncated moment equations can be written as an underdetermined linear system $\bm{a}_{\bm{\alpha}}^{\top} \bm{m} = 0\   (|\bm{\alpha}| \le \mu)$, where 
$\mu$ is the truncation order of the moment equation, $\bm{a}_{\bm{\alpha}}^{\top}$ is a multi-indexed constant vector determined from eq. \eqref{momeq1}, and $\bm{m}$ is the variables corresponding to all moments up to degree $\mu+1$. 

Moreover, the non-negativity of the probability and the support $\mathcal{X}$ of the probability distribution leads to additional necessary conditions
\begin{align}
    &\mathbb{E}[\bm{v}(\mathcal{M}_{\lceil \mu/2 \rceil})\bm{v}(\mathcal{M}_{\lceil \mu/2 \rceil})^{\top}] \succeq O,\\   
    &\mathbb{E}[x_k\bm{v}(\mathcal{M}_{\lceil (\mu-1)/2 \rceil})\bm{v}(\mathcal{M}_{\lceil (\mu-1)/2 \rceil})^{\top}] \succeq O,
\end{align}
for $k=1,2,\ldots, n$ (see also Notations in Section 1).
Using the variable $\bm{m}$, these conditions are more formally written as 
\begin{align}
&M(\bm{m}):= \mathcal{L}(\bm{v}(\mathcal{M}_{\lceil \mu/2 \rceil})\bm{v}(\mathcal{M}_{\lceil \mu/2 \rceil})^{\top})\succeq O, \label{eq:mom_mat}\\ 
&N^{(k)}(\bm{m}) := \mathcal{L}(x_k\bm{v}(\mathcal{M}_{\lceil (\mu-1)/2 \rceil})\bm{v}(\mathcal{M}_{\lceil (\mu-1)/2 \rceil})^{\top}) \succeq O,\label{eq:loc_mom_mat}
\end{align}
for $k=1,2,\ldots,n$, where $\mathcal{L}: \mathbb{R}[\bm{x}] \rightarrow \mathbb{R}$ is a linear functional, applied entrywise, such that 
\begin{equation}\label{eq:def_L}
  (\bm{m})_{\bm{\alpha}} = \mathcal{L}(\bm{x}^{\bm{\alpha}}),
\end{equation}
where $(\bm{m})_{\bm{\alpha}}$ denotes the entry of $\bm{m}$ associated with the multi-index $\bm{\alpha}$.
These observations lead to the following optimization program for bounding the stationary moments. 
 
\begin{theorem}\label{thm:con_momopt}\cite{sakurai2017convex,sakurai2018optimization}
Consider a stochastic chemical reaction network with given
stoichiometry $\{\bm{s}_i\}_{i=1}^{r}$ and propensity functions $\{w_{i}(\bm{x})\}_{i=1}^{r}$. Suppose the reaction network satisfies Assumption 1,
and let $\bm{m}^*$ denote the vector of the stationary moments. 
The optimal value of the following maximization (minimization, resp.) problem $\rho_{\max}$ ($\rho_{\min}$, resp.) satisfies $\rho_{\min} \le \bm{c}^\top \bm{m}^* \le \rho_{\max}$ for a given constant vector $\bm{c}$.
\begin{align}
\rho_{\max} = &\max_{\bm{m}} \ \bm{c}^\top \bm{m}\ \ (\rho_{\min} = \min_{\bm{m}} \ \bm{c}^\top \bm{m}) \\ &\text{subject to}\notag \\
&    0=\bm{a}_{\bm{\alpha}}^{\top}\bm{m} \quad (|\bm{\alpha}|\leq \mu)\label{eq:momeq_m}\\
& M(\bm{m}) \succeq O, N^{(k)}(\bm{m}) \succeq O \ (k=1,2,\cdots,n),
\end{align}
\end{theorem}
The optimization problem in Theorem 1 can be formulated as a semidefinite program (SDP). In general, increasing the truncation order $\mu$ provides progressively tighter bounds. 
However, the sizes of the \textit{moment matrix} $M$ and the \textit{localizing matrices} $N^{(k)} (k=1,2,\ldots,n)$ grow combinatorially with the number of molecular species $n$ and the truncation order $\mu$ as can be seen from \eqref{eq:mom_mat} and \eqref{eq:loc_mom_mat}. 
As a result, the computation cost rapidly becomes prohibitive, limiting the range of reaction systems that can be analyzed in this framework. 
This motivates the development of computation reduction techniques for the SDP in Theorem 1. 
\begin{problem}\label{pb:mot_proposed}
Consider the optimization problem in Theorem 1. Construct a reformulation that reduces the dimensions of the semidefinite constraints by exploiting sparsity structures induced by the moment equations. 
\end{problem}

\section{Matrix decomposition using reaction-wise sparsity}
In this section, we propose 
a sparsity-exploiting matrix decomposition of the moment and the localizing matrices, $M$ and $N^{(k)}~(k=1,2,\ldots,n)$.
By analyzing the moment equations on a reaction-wise basis, we show that each reaction involves only a limited subset of moments in these matrices. 
Exploiting this structure, 
the semidefinite constraints are reformulated using smaller matrices, resulting in a collection of reduced-size semidefinite constraints. 

\subsection{Reaction-wise sparsity of moment equations}
To reveal the reaction-wise sparsity structure, 
we first rewrite the stationary moment equation $0=\bm{a}_{\bm{\alpha}}^{\top}\bm{m}$ in \eqref{eq:momeq_m} as a sum of individual reaction components. This formulation directly follows from \eqref{momeq1}, where each term is associated with a reaction $\mathcal{R}_i$. 
We 
then express them in an equivalent matrix form using Frobenius inner products to relate the equality constraints to the moment and localizing matrices $M$ and $N^{(k)}$. 
Specifically, 
\begin{align}\label{eq:momdouti_reactionwise}
  0 = 
  \bm{a}_{\bm{\alpha}}^{\top}\bm{m} =
  \sum_{i=1}^{r} \langle A_{i}(\bm{\alpha}), M \rangle
  + \sum_{i=1}^{r}\sum_{k\in \mathcal{K}_i} \langle A^{(k)}_{i}(\bm{\alpha}), N^{(k)} \rangle,
\end{align}
where the set $\mathcal{K}_i$ denotes the collection of indices corresponding to the reactant species involved in reaction $\mathcal{R}_i$, and
$A_i$ and 
$A_{i}^{(k)}$ are coefficient matrices determined by $w_i(\bm{x})$ and $\bm{s}_i$.
It should be recalled that equality constraints are expressed using the trace inner product form in the standard SDP formulation. 

In general, the coefficients $A_{i}$ and $A_{i}^{(k)}$ are not uniquely determined. 
However, a key observation here is that there are entries in $A_{i}$ and $A_{i}^{(k)}$ that are necessarily zero since each reaction involves at most two molecular species as reactants and $w_i(\bm{x})$ is the polynomial of at most two variables (Table \ref{propent}). 
As shown in the next subsection, this structural sparsity allows an equivalent reformulation of eq. \eqref{eq:momdouti_reactionwise} in which entries corresponding to the structurally zero components are omitted, leading to smaller semidefinite matrices. 
The following theorem characterizes a set of moments that appear in each term of eq. \eqref{momeq1} up to truncation order $\mu$..

\begin{theorem}\label{thm:seiseikou}
Consider the $i$-th term of the moment equation for 
$\mathbb{E}[\bm{x}^{\bm{\alpha}}]$ given in 
eq.~\eqref{momeq1}, \textit{i.e.,} 
$f_{i, \bm{\alpha}} := \mathbb{E}[\{(\bm{x}+\bm{s}_i)^{\bm{\alpha}}-\bm{x}^{\bm{\alpha}}\}w_i(\bm{x})]$. 
A moment $\mathbb{E}[\bm{x}^{\bm{\eta}}]$ appears in $f_{i,\bm{\alpha}}$ for some $\bm{\alpha}$ with $|\bm{\alpha}| \le \mu$
if and only if $\bm{x}^{\bm{\eta}} \in \Xi_i$, where 
\begin{align}\label{eq:def_Xi_i}
   \Xi_{i} := \left\{ \bm{x}^{\bm{\beta}+\bm{\gamma}} ~\middle|~
   \begin{array}{l}
         \bm{\beta}, \bm{\gamma}\in \mathbb{N}_0^n, \ 0\leq|\bm{\beta}| \leq \mu-1, \\
         1\leq\gamma_j\leq d_j \ (j\in \mathcal{K}_i),\\
        \gamma_j=0 \ (j\notin \mathcal{K}_i) \\
   \end{array}
   \right\},
\end{align}
and $d_j$ 
is the $j$-th component of the multi-index $\bm{d} \in \mathbb{N}_0^n$ for the highest-degree monomial $\bm{x}^{\bm{d}}$ in the propensity function $w_i(\bm{x})$.
\end{theorem}

\proof
     The stoichiometric vector $\bm{s}_i$ represents the change in the copy number of each molecular species by reaction $\mathcal{R}_i$. 
Hence, only the components corresponding to the reactant or product species of $\mathcal{R}_i$ can be nonzero.
Thus, 
     \begin{align}
    (\bm{x}+\bm{s}_i)^{\bm{\alpha}}-\bm{x}^{\bm{\alpha}}= \prod_{j\in \mathcal{S}_i}&(x_j+s_{ij})^{\alpha_j}\prod_{j\notin \mathcal{S}_i}x_j^{\alpha_j} - \bm{x}^{\bm{\alpha}}\notag
\end{align}
holds, where $\mathcal{S}_i = \{j \mid s_{ij} \neq 0\}$. Moreover, 
$w_i(\bm{x})$ is one of the functions shown in Table \ref{propent}.
Consequently, all monomials in $f_{i,\bm{\alpha}}$ can be identified as 
\begin{align}
   \Xi_{\bm{\alpha},i} :=  \left\{ \bm{x}^{\bm{\beta}+\bm{\gamma}} ~\middle|~ 
   \begin{array}{l}
         \bm{\beta}, \bm{\gamma}\in \mathbb{N}_0^n, \ 0\leq|\bm{\beta}| \leq |\bm{\alpha}|-1,\\
         0\leq\beta_j \leq \alpha_j \ (j \in \mathcal{S}_i), \\
        \beta_j = \alpha_j \ (j \notin \mathcal{S}_i),\\
        1\leq\gamma_j\leq d_j\  (j\in \mathcal{K}_i), \\
        \gamma_j=0 \ (j\notin \mathcal{K}_i)\\
   \end{array}
   \right\}.\notag
\end{align}
For a fixed $i$, all moments  appeared in $f_{i, \bm{\alpha}}$ with $|\bm{\alpha}| \le \mu$ is given $\mathbb{E}[\bm{x}^{\bm{\eta}}]$ with $\bm{x}^{\bm{\eta}} \in \bigcup_{|\bm{\alpha}| \leq \mu}\Xi_{\bm{\alpha},i}$. 
In this union, the condition on $\bm{\gamma}$ is independent of $\bm{\alpha}$, while the condition on $\bm{\beta}$ generates all monomials up to degree $\mu-1$. Therefore, we have
$\Xi_{i} = \bigcup_{|\bm{\alpha}| \leq \mu}\Xi_{\bm{\alpha},i}$, which is eq. \eqref{eq:def_Xi_i}.
\endproof

\medskip
This theorem implies that only a subset of moments appear in the $i$-th term of the moment equation \eqref{eq:momdouti_reactionwise} depending on $w_i(\bm{x})$. 
In particular, eq. \eqref{eq:def_Xi_i} implies the existence of a set of monomials $\bar{\Phi}_i \subseteq \mathcal{M}_{\lceil \mu/2 \rceil}$ such that 
\begin{align}\label{eq:const_bar_Phi}
z_1, z_2 \in \bar{\Phi}_i \implies 
z_1 z_2 \notin \Xi_i.
\end{align}
We define $\bar{\Phi}_i$ as the largest set satisfying \eqref{eq:const_bar_Phi} and $\Phi_i :=  \mathcal{M}_{\lceil \mu/2 \rceil} \backslash \bar{\Phi}_i$, where a constructive definition is also shown in Appendix \ref{app:def_Phi_i}. 
Based on this partition of the monomial set $\mathcal{M}_{\lceil \mu/2 \rceil}$, we permute the moment matrix $M$ as 
\begin{align}
\hat{M}_i(\bm{m}) := \mathcal{L}\left[
\begin{bmatrix}
\bm{v}(\Phi_i) \\
\bm{v}(\bar{\Phi}_i)
\end{bmatrix}
\begin{bmatrix}
\bm{v}^\top(\Phi_i) & \bm{v}^\top(\bar{\Phi}_i)
\end{bmatrix}
\right]\notag
\end{align}
by reordering the monomial basis. 
An important observation is that, by definition, the entries of the lower-right block $\bm{v}(\bar{\Phi}_i)\bm{v}(\bar{\Phi}_i)^\top$ do not appear in the $i$-th term of the moment equation. 
This \textit{sparsity} of the matrix $\hat{M}_i$ will be a key to the reduction of the moment matrix in the optimization constraint. 
Similarly, we define a permuted matrix of $\hat{N}^{(k)}$ by 
\begin{align}
\hat{N}_i^{(k)}(\bm{m}) := \mathcal{L}\left[
\begin{bmatrix}
\bm{v}(\Omega_i^{(k)}) \\
\bm{v}(\bar{\Omega}_i^{(k)})
\end{bmatrix}
\begin{bmatrix}
\bm{v}^\top(\Omega_i^{(k)}) & \bm{v}^\top(\bar{\Omega}_i^{(k)})
\end{bmatrix}
\right],\notag
\end{align}
where $\bar{\Omega}_i^{(k)} \subseteq \mathcal{M}_{\lceil (\mu-1)/2 \rceil}$ is the largest set such that
\begin{align}\label{eq:const_bar_Omega}
z_1, z_2 \in \bar{\Omega}_{i}^{(k)} \implies 
x_kz_1 z_2 \notin \Xi_i
\end{align}
holds, and $\Omega_i^{(k)} := \mathcal{M}_{\lceil (\mu-1)/2 \rceil} \setminus \bar{\Omega}_i^{(k)}$ (see Appendix \ref{app:def_Omega_i} for a constructive definition).
\subsection{Sparsity-based reformulation of the constraints}

Based on the permuted matrices defined 
in the previous subsection, we will now formulate 
a computationally efficient relaxation problem. 
For this purpose, we consider a partition of the monomial set $\bar{\Phi}_i$ and $\bar{\Omega}_i^{(k)}$ such that 
\begin{align}\label{eq:set_def_bar}
\bar{\Phi}_i = \bigcup_{j=1}^{p_i} \bar{\Phi}_{i,j} \quad \mathrm{and} \quad \bar{\Omega}_i^{(k)} = \bigcup_{j=1}^{q_i} \bar{\Omega}_{i,j}^{(k)},
\end{align}
respectively. 
Using these partitioned monomial sets $\bar{\Phi}_{i,j}$ and $\bar{\Omega}_{i,j}^{(k)}$, we define reduced moment matrices by 
\begin{align}
&\hat{M}_{i,j} :=  \mathcal{L}(  \bm{v}(\Phi_i \cup \bar{\Phi}_{i,j})  \bm{v}(\Phi_i \cup \bar{\Phi}_{i,j})^\top )  \label{eq:hatMij}\\
&\hat{N}^{(k)}_{i,j} :=  \mathcal{L}(x_k\bm{v}(\Omega_i^{(k)} \cup \bar{\Omega}_{i,j}^{(k)})  \bm{v}(\Omega_i^{(k)} \cup \bar{\Omega}_{i,j}^{(k)})^\top ), \label{eq:hatNij}
\end{align}
where $j=1,2\ldots,p_i$ for \eqref{eq:hatMij}, and $j=1,2\ldots,q_i$ and $k \in \mathcal{K}_i$ for \eqref{eq:hatNij}.

The following lemma shows that the $i$-th terms of the moment equation \eqref{eq:momdouti_reactionwise} can be represented by these reduced matrices.

\begin{lemma}\label{lemma:equality_coef}
Consider the reformulated moment equation \eqref{eq:momdouti_reactionwise}. Let $\hat{M}_{i,j} \in \mathbb{R}^{(|\Phi_i|+|\bar{\Phi}_{i,j}|)\times(|\Phi_i|+|\bar{\Phi}_{i,j}|)}$ and $\hat{N}_{i,j}^{(k)} \in \mathbb{R}^{(|\Omega_i^{(k)}| + |\Omega_{i,j}^{(k)}|) \times (|\Omega_i^{(k)}| + |\Omega_{i,j}^{(k)}|)}$ be given by eqs. \eqref{eq:hatMij} and \eqref{eq:hatNij}, respectively.  
Then, there exists coefficient matrices $\hat{A}_{i,j}(\bm{\alpha})$ and $\hat{A}^{(k)}_{i,j}(\bm{\alpha})$ such that the following equalities hold.
\begin{align}
&    \langle A_{i}(\bm{\alpha}), M \rangle = \sum_{j=1}^{p_i} \langle \hat{A}_{i,j}(\bm{\alpha}), \hat{M}_{i,j} \rangle,\label{eq:equi_tran_mom}\\
&\langle A^{(k)}_{i}(\bm{\alpha}), N^{(k)} \rangle = \sum_{j=1}^{q_i} \langle \hat{A}_{i,j}^{(k)}(\bm{\alpha}), \hat{N}_{i,j}^{(k)} \rangle\label{eq:equi_tran_loc}
\end{align}
\end{lemma}

\proof
   The matrix $\hat{M}_{i}$ is obtained from the matrix $M$ by permuting the monomial basis. 
   Moreover, since the entries $\mathcal{L}(\bm{v}(\bar{\Phi}_i)\bm{v}(\bar{\Phi}_i)^\top)$ in $\hat{M}_i$ do not appear in the moment equation, there exists $\hat{A}_i(\bm{\alpha})$ whose bottom-right block is a zero matrix, satisfying 
    \begin{align}
        &\langle A_{i}(\bm{\alpha}), M \rangle = \langle \hat{A}_{i}(\bm{\alpha}), \hat{M}_i \rangle, \label{eq:equ_hatM_i}\\
        &\hat{A}_{i}(\bm{\alpha}) = 
        \begin{bmatrix}
           {\hat{A}}^i_{\bm{v}(\Phi_i),\bm{v}(\Phi_i)} & \bar{A}^i_{\bm{v}(\Phi_i),\bm{v}(\bar{\Phi}_{i})}\\
           \bar{A}^i_{\bm{v}(\bar{\Phi}_{i}),\bm{v}(\Phi_i)} & O
       \end{bmatrix}\label{eq:blo_arr_pattern},
    \end{align}
    where $\hat{A}^i_{\bm{p},\bm{q}}$ is a coefficient matrix composed of the entries of $\hat{A}_i$ that correspond to $\mathcal{L}(\bm{p}\bm{q}^\top)$, a subset of the elements in matrix $M$.
    Further, 
    it follows from eq.~\eqref{eq:blo_arr_pattern} that 
    \begin{align}
        \qty{z_1z_2 \middle| z_1 \in \Phi_i, z_2\in \bar{\Phi}_{i}} = \bigcup_{j=1}^{p_i} \qty{z_1z_2 \middle| z_1 \in \Phi_i, z_2\in \bar{\Phi}_{i,j}}.\notag
    \end{align}
Hence, for $\hat{M}_i$ and $\hat{M}_{i,j}$, there exists $\hat{A}_{i,j}(\bm{\alpha})$ satisfying eq. \eqref{eq:equi_tran_mom}.
The equality \eqref{eq:equi_tran_loc} follows from the same argument for $\hat{N}_i^{(k)}$ and $\hat{N}_{i,j}^{(k)}$. 
\endproof

\medskip
This lemma implies that the equality constraints of the original optimization problem in Theorem \ref{thm:con_momopt} can be equivalently formulated using the smaller moment matrices. 
\begin{figure}[t] 
  \begin{center}
    \includegraphics[width=\linewidth]{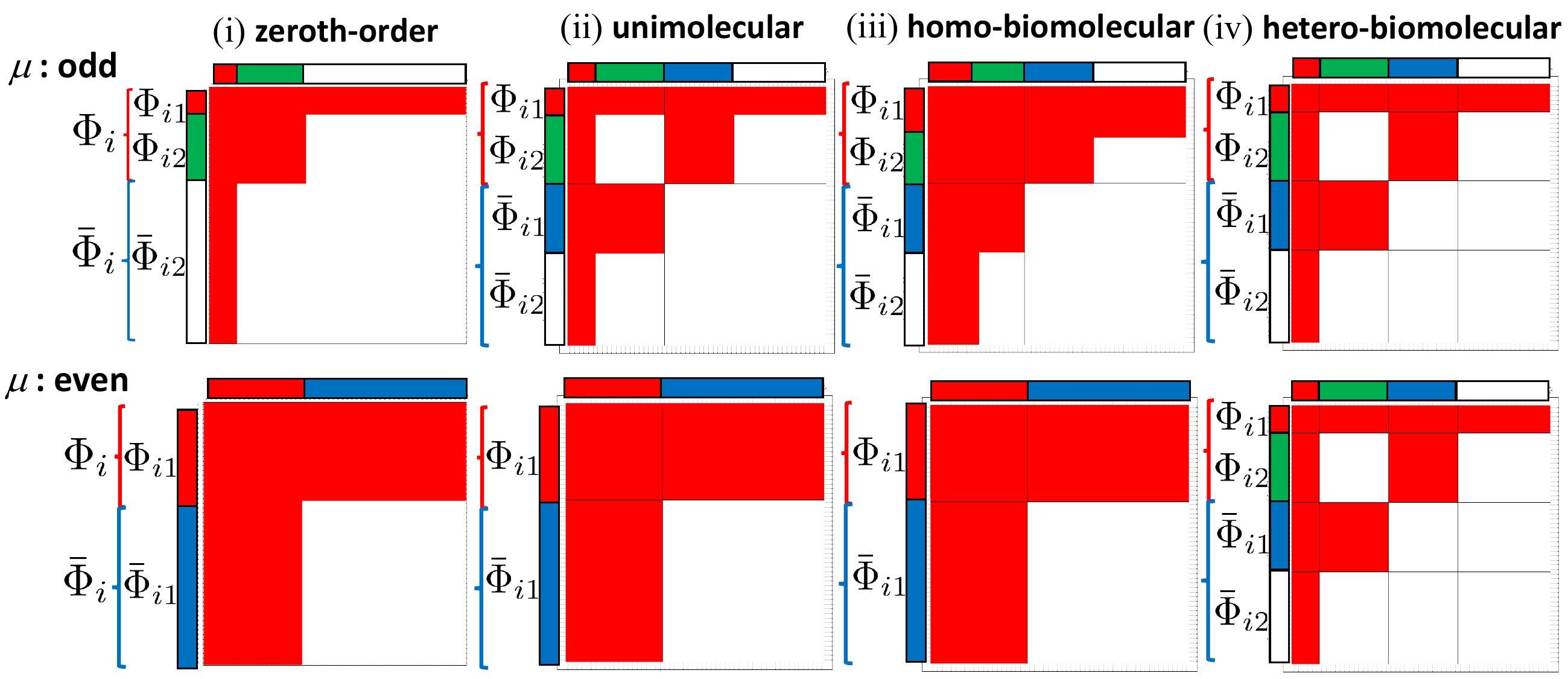}
  \end{center}
  \caption{Detailed sparsity pattern of the coefficient matrix $\hat{A}_i$ in eq. \eqref{eq:equ_hatM_i}}
  \label{fig:ASP_RW}
\end{figure}
\begin{remark}\label{remark:Phi_i1i2}
 A similar relationship holds for the cross moments constructed by $\Phi_i$ and $\bar{\Phi}_i$ as in eq.~\eqref{eq:const_bar_Phi}. Specifically, 
Let $\Phi_{i1}$ and $\bar{\Phi}_{i1}$ be the minimal monomial sets such that their complement sets, defined by $\Phi_{i2} = \Phi_i \setminus \Phi_{i1}$ and $\bar{\Phi}_{i2} = \bar{\Phi}_i \setminus \bar{\Phi}_{i1}$, satisfy
\begin{align}
    z_1 \in \Phi_{i2},  z_2 \in \bar{\Phi}_{i2} \implies z_1z_2 \notin \Xi_i. \label{eq:i2_const}
\end{align}
    Figure~\ref{fig:ASP_RW} illustrates the detailed sparsity pattern of $\hat{A}_i$ in eq. \eqref{eq:equ_hatM_i} based on these definitions. In the numerical demonstration in Section \ref{sec:num_dem}, partitioning of the monomial set for $\hat{M}_{i,j}$ is given based on 
    $\bar{\Phi}_{i1,j} (j=1,2,\cdots, p_{i1})$ and $\bar{\Phi}_{i2,j}(j=1,2,\cdots, p_{i2})$ associated with $\bar{\Phi}_{i1}$ and $\bar{\Phi}_{i2}$, respectively.
    Then, $p_{i1}$ and $p_{i2}$ are given such that $p_i \leq p_{i1}+p_{i2}$.
\end{remark}

\medskip
\noindent
\textbf{Example.}
Consider a dimerization reaction ($n=2$),
\begin{align}
  \emptyset \xrightarrow{\theta_1} X_1 \quad X_1 \xrightarrow{\theta_2} \emptyset \quad X_1+X_1\xrightarrow{\theta_3} X_2 \quad X_2 \xrightarrow{\theta_4} \emptyset,\notag 
\end{align}
and its truncated moment equations. 
For $\mu = 3$, the monomial basis for the moment matrix is $\bm{v}_{\lceil{3/2}\rceil} = 
[1, x_1, x_2, x_1^2, x_1x_2, x_2^2]^\top$. 
Theorem 2 implies that not all combinations of the monomials $\bm{v}_{\lceil{3/2}\rceil}$ appear in the $i$-th term of the moment equation \eqref{eq:momdouti_reactionwise}. For example, 
for the dimerization reaction $\mathcal{R}_3$, the third term of the moment equation in \eqref{eq:momdouti_reactionwise} is expressed as $\langle A_{3}, M \rangle$, where, for $\bm{\alpha}= [2,1]^{\top}$,
\begin{align}
 A_3(\bm{\alpha}) = \theta_3
 \begin{bmatrix}
     0 & 2 & 0 & -2 & -1 & 0\\
     2 & -4 & -1 & \frac{5}{2} & 2 & 0\\
     0 & -1 & 0 & 2 & 0 & 0\\
     -2 & \frac{5}{2} & 2 & -1 & -2 & 0\\
     -1 & 2 & 0 & -2 & 0 & 0\\
     0 & 0 & 0 & 0 & 0 & 0
 \end{bmatrix}, 
\end{align}
and $M$ is defined in \eqref{eq:mom_mat}. 
Thus, the unused monomial set is given by $\bar{\Phi}_3 = \{1, x_2, x_2^2$\}.
Partitioning $\bar{\Phi}_3$ into 
$\bar{\Phi}_{3,1} = \{1, x_2\}$ and $\bar{\Phi}_{3,2} = \{x_2^2\}$, we can rewrite $\langle A_{3}, M \rangle$ using smaller matrices 
$\hat{M}_{3,1} \in \mathbb{R}^{5 \times 5}$ as  
$\hat{M}_{3,2} \in \mathbb{R}^{4 \times 4}$ as 
$\langle A_{3}, M \rangle= \langle \hat{A}_{3,1}, \hat{M}_{3,1}\rangle 
 + \langle \hat{A}_{3,2}, \hat{M}_{3,2} \rangle$, where 
 \begin{align}
 &\hat{A}_{3,1}(\bm{\alpha}) = \theta_3
 \begin{bmatrix}
     0 & 2 & 0 & -2 & -1\\
     2 & -2 & -1 & \frac{5}{4} & 1\\
     0 & -1 & 0 & 2 & 0\\
     -2 & \frac{5}{4} & 2 & -\frac{1}{2} & -1\\
     -1 & 1 & 0 & -1 & 0
 \end{bmatrix}, \\
 &\hat{A}_{3,2}(\bm{\alpha}) =\theta_3
 \begin{bmatrix}
     -2 & \frac{5}{4} & 1 & 0\\
     \frac{5}{4} & -\frac{1}{2} & -1 & 0\\
     1 & -1 & 0 & 0\\
     0 & 0 & 0 & 0
 \end{bmatrix}
 .
 \end{align}
 $\hfill \Box$

The last key observation is the following lemma. 

\begin{lemma}\label{lemma:PSD_nec}
For each $i=1,2,\ldots,r$, the following  conditions hold.
\begin{align}
&    M \succeq O \implies \hat{M}_{i,j} \succeq O \ \  (j=1,2\ldots,p_i),\\
&    N^{(k)} \succeq O \implies \hat{N}^{(k)}_{i,j} \succeq O \ \ (j=1,2\ldots,q_i)
\end{align}
\end{lemma}
The lemma holds because $\hat{M}_{i,j}$ and $\hat{N}^{(k)}_{i,j}$ are principal submatrices of $M$ and $N^{(k)}$, respectively. 
Using Lemmas \ref{lemma:equality_coef} and \ref{lemma:PSD_nec}, we can now formulate a relaxed optimization problem with smaller semidefinite matrices, which leads a computationally efficient SDP.  

\begin{theorem}\label{thm:proposed_opt}
Consider a stochastic chemical reaction network with given
stoichiometry vectors $\{\bm{s}_i\}_{i=1}^{r}$ and propensity functions $\{w_{i}(\bm{x})\}_{i=1}^{r}$. Suppose the reaction network satisfies Assumption 1, 
and let $\bm{m}^*$ denote the vector of the stationary moments. 
The optimal value of the following maximization (minimization, resp.) problem $\hat{\rho}_{\max}$ ($\hat{\rho}_{\min}$, resp.) satisfies $\hat{\rho}_{\min} \le \bm{c}^\top \bm{m}^* \le \hat{\rho}_{\max}$ for a given constant $\bm{c}$.
\begin{align}
&\hat{\rho}_{\max} = \max_{\bm{m}} \ \bm{c}^\top \bm{m}\ \ (\hat{\rho}_{\min} = \min_{\bm{m}} \ \bm{c}^\top \bm{m}) \\ &\text{subject to}\notag \\
&    0=\sum_{i=1}^{r}\qty(\sum_{j=1}^{p_i} \langle \hat{A}_{i,j}(\bm{\alpha}), \hat{M}_{i,j} \rangle+\sum_{k\in \mathcal{K}_i}\sum_{j=1}^{q_i}\langle\hat{A}_{i,j}^{(k)}(\bm{\alpha}), \hat{N}_{i,j}^{(k)}\rangle)\notag \\
&\hspace{160pt}  (|\bm{\alpha}|\leq \mu)\label{eq:momeq_RW}\\
& \hat{M}_{i,j} \succeq O \ (i=1,2,\cdots,r, \ j=1,2,\cdots,p_i)\label{eq:mommatpsd_RW}\\
& \hat{N}^{(k)}_{i,j} \succeq O \ (i=1,2,\cdots,r,j=1,2,\cdots,q_i, 
k \in \mathcal{K}_i)\label{eq:locmommatpsd_RW}
\end{align}
\end{theorem}

\proof
The equivalence between eq. \eqref{eq:momeq_m} and \eqref{eq:momeq_RW} directly follows from Lemma \ref{lemma:equality_coef}. 
Further, the semidefinite constraints \eqref{eq:mom_mat} and \eqref{eq:loc_mom_mat} imply those in eq. \eqref{eq:mommatpsd_RW} and \eqref{eq:locmommatpsd_RW} (Lemma \ref{lemma:PSD_nec}). Therefore, the feasible set of the optimization problem in Theorem 2 contains that in Theorem 1. 
Since $\bm{m}^*$ in Theorem 1 is feasible, we have $\hat{\rho}_{\min} \le \bm{c}^\top \bm{m}^* \le \hat{\rho}_{\max}$.
\endproof
\medskip

Table \ref{table:matrix_size_number} shows the size of the semidefinite matrices and the number of semidefinite constraints in the reduced optimization problem in Theorem 3, where $|\bar{\Phi}_{i,j}| \le |\bar{\Phi}_i|$ and $|\bar{\Omega}^{(k)}_{i,j}| \le |\bar{\Omega}^{(k)}_i|$ hold due to the partition \eqref{eq:set_def_bar}.
Since the dominant computational cost of SDP grows cubically with the matrix dimension \cite{nesterov2013introductory}, this reformulation leads to improved computational efficiency.
\begin{table}[tb]
    \centering
    \caption{Size and number of matrix constraints}
    \resizebox{\linewidth}{!}{
    \begin{tabular}{cccc}
    \hline
        Designation & Matrix &  Size & Number\\
       \hline \hline
         Moment matrix & $M$ & $|\Phi_i|+ |\bar{\Phi}_i|$ & 1\\
        Decomposed moment matrix &$\hat{M}_{i,j}$ & $|\Phi_i|+ |\bar{\Phi}_{i,j}|$ & $p_i$\\
        Localizing matrix & $N^{(k)}$ & $|\Omega^{(k)}_i|+ |\bar{\Omega}^{(k)}_i|$ & 1 \\
         Decomposed localizing matrix& $\hat{N}_{i,j}^{(k)}$ & $|\Omega^{(k)}_i|+ |\bar{\Omega}^{(k)}_{i,j}|$& $q_i$ \\
        \hline
    \end{tabular}
    }
    \label{table:matrix_size_number}
\end{table}
The proposed optimization problem is a relaxation of the original one, and thus, the resulting bounds may become
less tight, i.e., $\hat{\rho}_{\min} \le \rho_{\min}$ and
$\rho_{\max} \le \hat{\rho}_{\max}$. However, as shown in
Section~\ref{sec:chordal}, the proposed decomposition is closely related to
chordal decomposition used in sparse semidefinite programming, which typically provides reasonably tight bounds \cite{fukuda2001exploiting,wang2021chordal,zheng2021chordal}. 
Thus, the conservativeness introduced by this relaxation is expected to be limited.

\section{Numerical Demonstration}\label{sec:num_dem}
We consider the negative feedback gene regulatory system with DNA sponge \cite{wan2020synthetic},  which sequesters the repressor protein (Protein 2), as shown in  Fig.~\ref{fig:7var_RW_bound_time}(a). This system consists of $n=7$ molecular species and $r=14$ reactions. 
Our goal here is to 
find the steady-state mean copy number of
Protein 2, $\mathbb{E}[x_6]$, and to compare the computation times of the original optimization in Theorem \ref{thm:con_momopt} and the proposed one in Theorem \ref{thm:proposed_opt}. 

For this purpose, we formulated the optimization problems for various truncation orders $\mu$ of the moment equation. 
The optimal values obtained from these optimization problems, shown in Figure \ref{fig:7var_RW_bound_time}(b), illustrate the convergence of the bounds as $\mu$ increases. 
Since the proposed optimization is a relaxation of the original problem, the resulting bounds tend to be looser. However, the gap between the upper and lower bounds remains the order of $10^{-2}$ for $\mu=6$. In particular, for $\mu=5, 6$, the bounds obtained by the original and the proposed approach are comparable, indicating that the proposed method provides practically useful bounds of the stationary moments.

As expected from Theorem 3, the maximum dimension of semidefinite matrices involved in each optimization is smaller for the proposed approach (see Table \ref{table:const_7var}), which contributes to the reduction in computation cost. 
Specifically, for $\mu=5, 6$ the proposed method reduces computation time by approximately $20\%$ as shown in Fig. \ref{fig:7var_RW_bound_time}(c)
for $\mu= 5, 6$. 
These results indicate that, by reducing the number of variables, the proposed method can efficiently provide practically useful bounds with accuracy comparable to that of the original method.

The rate constants used in this simulation is summarized in Appendix \ref{sec:rateconst}. All computations were implemented in Julia using the Clarabel optimization solver\cite{goulart2024clarabel}.
\begin{table}[tb]
\centering

\caption{Problem size for each truncation order $\mu$, showing the maximum dimension of semidefinite matrices, the number of constraints attaining this maximum dimension, and the total number of constraints. Values in parentheses indicate the corresponding numbers for the original problem (Theorem 1).}
\begin{tabular}{cccc}  \hline
$\mu$ & Max matrix dim. & \# Max-dim constraint & \# constraint \\ \hline
1 & 2 (8) & 10 (1) & 23 (8)\\
2 & 2 (8) & 161 (8) & 161 (8)\\
3 & 10 (36) & 30 (1) & 191 (8)\\
4 & 10 (36) & 345 (8) & 429 (8)\\
5 & 48 (120) & 20 (1) & 420 (8)\\
6 & 48 (120) & 174 (8) & 306 (8)\\ \hline
\end{tabular}
\label{table:const_7var}
\end{table}

\begin{remark}
Among the reactions shown in Fig.~\ref{fig:7var_RW_bound_time}(a), there exist reactions
whose propensity functions $w_i(\bm{x})$ differ from those of Table \ref{propent}.
In such cases, the matrix decomposition is performed by treating them as two
distinct reactions that share the same stoichiometric vector
$\bm{s}_i$.
\end{remark}

\begin{remark}
The size of each partitioned monomial set $\bar{\Phi}_{i1,j}$, $\bar{\Phi}_{i2,j}$ (see Remark \ref{remark:Phi_i1i2})  is determined so as to balance computational efficiency and redundancy in the resulting semidefinite constraints.
In general, such decompositions may lead to degraded computational performance when there is significant overlap among the partitioned sets. A common strategy to mitigate this issue is clique merging \cite{garstka2020clique}, which reduces redundancy by appropriately enlarging overlapping blocks.
Motivated by this idea, the numbers of elements in the appended sets are determined by $|\bar{\Phi}_{i1,j}| = \ell(\Phi_{i})$, $|\bar{\Phi}_{i2,j}| = \ell(\Phi_{i1})$, and $|\bar{\Omega}^{(k)}_{i,j}| = \ell(\Omega_{i}^{(k)})$, where the function $\ell(\mathcal{S})$ for a base set $\mathcal{S}$ is defined as
\begin{align}
    \ell(\mathcal{S}) := \max\{\ell \in \mathbb{N} \mid 2(|\mathcal{S}|+\ell)^3  -(|\mathcal{S}|+2\ell)^3 \geq 0\}. \notag
\end{align}
\end{remark}

\begin{figure}
  \centering
  \includegraphics[width=\linewidth]{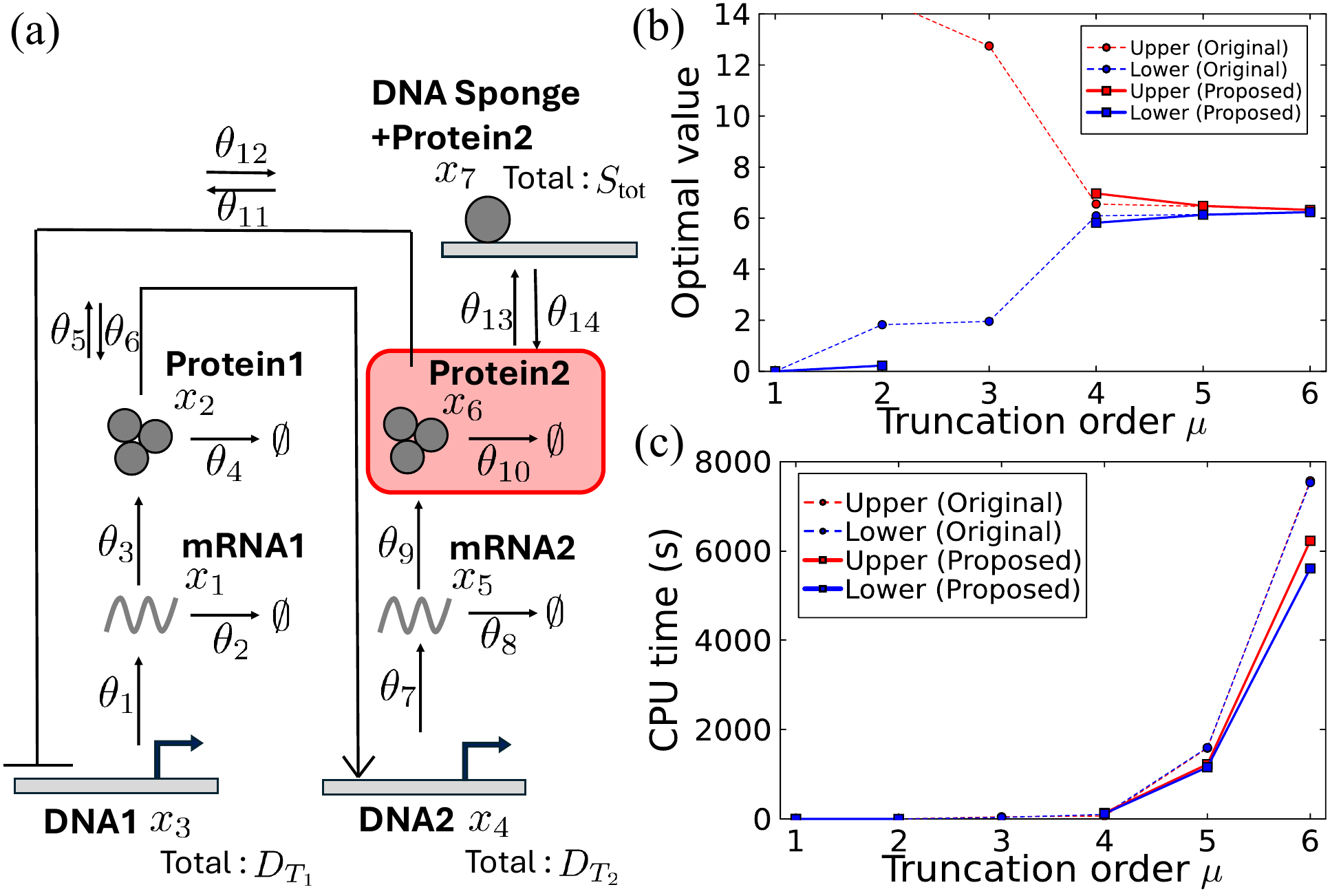}
  \caption{Comparison before and after matrix decomposition for each truncation order $\mu$ (a) Schematic diagram of the reactions ($\rightarrow$ : Activation, $\dashv$ : Repression) (b) the bounds of $\mathbb{E}[x_6]$ (c) computation time}
  \label{fig:7var_RW_bound_time}
\end{figure}

\section{Structural Interpretation via Chordal Decomposition}\label{sec:chordal}
\begin{figure}[t] 
  \begin{center}
    \includegraphics[width=\linewidth]{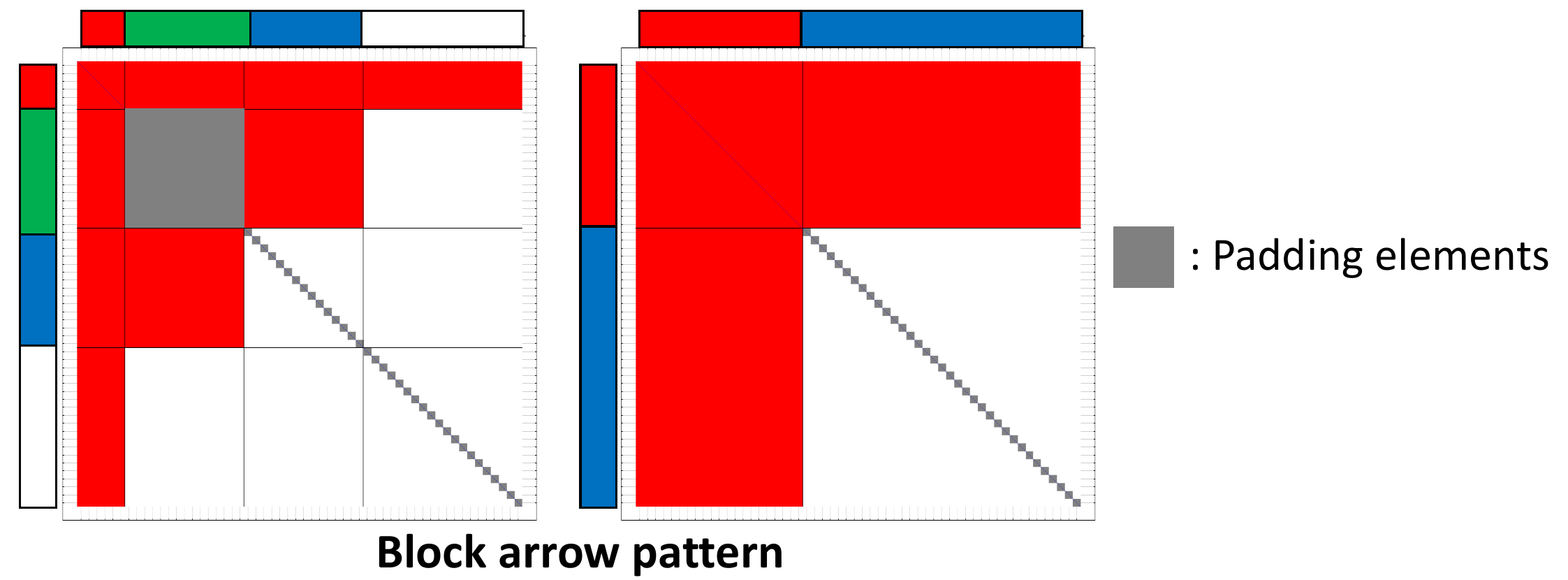}
  \end{center}
  \caption{Block arrow pattern of the coefficient matrices}
  \label{fig:Block_arrow}
\end{figure}
In this section, we provide a structural interpretation of the proposed relaxation in the context of chordal decomposition in semidefinite programming \cite{fukuda2001exploiting} and discuss why the resulting bounds are expected to remain reasonably tight in practice.

The sparsity structure induced by the proposed decomposition can be interpreted in terms of chordal sparsity, \textit{i.e.}, a graph structure for which positive semidefiniteness can be characterized exactly by specific, smaller overlapping submatrices.
To see this, we consider the interaction pattern among monomials induced by the sets $\Phi_i \cup \bar{\Phi}_{i,j}$ $(j=1,2,\ldots,p_i)$, and characterize this pattern through the nonzero structure of the coefficient matrix $\hat{A}_i$. 
This pattern corresponds to the entries shown in Fig.~\ref{fig:Block_arrow}. It represents a structure obtained by padding entries into the nonzero structure of the coefficient matrix $\hat{A}_i$ based on the permutation of the vector basis $\bm{v}(\mathcal{M}_{\lceil \mu/2 \rceil})$. This results in a block arrow pattern, which corresponds to a chordal graph.

In general, a symmetric matrix $X$ with a chordal sparsity pattern is positive semidefinite completable if and only if its principal submatrices $\{X_{\mathcal{C}}\}$ corresponding to the maximal cliques $\mathcal{C}$ of the associated graph are positive semidefinite, where \textit{positive semidefinite completable} means that the matrix $X$ can be made positive semidefinite by appropriately choosing the free entries\cite{grone1984positive}. 
In other words, enforcing positive semidefiniteness on the clique-wise submatrices is not only necessary but also sufficient to ensure positive semidefiniteness of the original matrix $X$ without introducing a relaxation.
In this sense, the sparsity structure induced by the proposed decomposition is consistent with a setting in which no relaxation would be introduced.

However, structural gaps remain between the proposed decomposition and the exact chordal decomposition, which prevent the equivalence from holding and thus introduce a relaxation. Specifically, the matrices $M$ and $N^{(k)}$ possess a structure induced by $\mathbb{E}[\bm{v}(\mathcal{M}_{\lceil \mu/2 \rceil})\bm{v}(\mathcal{M}_{\lceil \mu/2 \rceil})^{\top})]$ 
(see Eqs. \eqref{eq:mom_mat} and \eqref{eq:loc_mom_mat}), which enforces
algebraic dependencies among their entries. This structure is
not imposed in the positive semidefinite completability
problem. Consequently, Lemma~\ref{lemma:PSD_nec} provides only
a necessary condition.

\section{Conclusion}
In this paper, we have proposed a method to accelerate optimization computation by reducing the size of positive semidefinite matrix constraints in moment bounding optimization for stochastic chemical reactions. The proposed approach exploits the fact that each reaction involves only a limited subset of variables and reformulates the constraints using smaller moment matrices. 
We have shown that this decomposition is related to chordal decomposition, a standard sparsity-exploiting technique in semidefinite programming, and is therefore expected to yield reasonably tight bounds in practice despite introducing a relaxation.
The numerical example has demonstrated that the proposed formulation significantly reduces the size of the semidefinite constraints, thereby improving computational efficiency while providing practically useful bounds.

\bibliographystyle{IEEEtran} 

\begin{thebibliography}{10}
\providecommand{\url}[1]{#1}
\csname url@samestyle\endcsname
\providecommand{\newblock}{\relax}
\providecommand{\bibinfo}[2]{#2}
\providecommand{\BIBentrySTDinterwordspacing}{\spaceskip=0pt\relax}
\providecommand{\BIBentryALTinterwordstretchfactor}{4}
\providecommand{\BIBentryALTinterwordspacing}{\spaceskip=\fontdimen2\font plus
\BIBentryALTinterwordstretchfactor\fontdimen3\font minus \fontdimen4\font\relax}
\providecommand{\BIBforeignlanguage}[2]{{%
\expandafter\ifx\csname l@#1\endcsname\relax
\typeout{** WARNING: IEEEtran.bst: No hyphenation pattern has been}%
\typeout{** loaded for the language `#1'. Using the pattern for}%
\typeout{** the default language instead.}%
\else
\language=\csname l@#1\endcsname
\fi
#2}}
\providecommand{\BIBdecl}{\relax}
\BIBdecl

\bibitem{kitano2002systems}
H.~Kitano, ``Systems biology: a brief overview,'' \emph{Science}, vol. 295, no. 5560, pp. 1662--1664, 2002.

\bibitem{del2015biomolecular}
D.~Del~Vecchio and R.~M. Murray, \emph{Biomolecular feedback systems}.\hskip 1em plus 0.5em minus 0.4em\relax Princeton University Press Princeton, NJ, 2015.

\bibitem{elowitz2002stochastic}
M.~B. Elowitz, A.~J. Levine, E.~D. Siggia, and P.~S. Swain, ``Stochastic gene expression in a single cell,'' \emph{Science}, vol. 297, no. 5584, pp. 1183--1186, 2002.

\bibitem{gillespie1992rigorous}
D.~T. Gillespie, ``A rigorous derivation of the chemical master equation,'' \emph{Physica A: Statistical Mechanics and its Applications}, vol. 188, no.~1, pp. 404--425, 1992.

\bibitem{elf2003fast}
J.~Elf and M.~Ehrenberg, ``Fast evaluation of fluctuations in biochemical networks with the linear noise approximation,'' \emph{Genome research}, vol.~13, no.~11, pp. 2475--2484, 2003.

\bibitem{gillespie2000chemical}
D.~T. Gillespie, ``The chemical langevin equation,'' \emph{The Journal of Chemical Physics}, vol. 113, no.~1, pp. 297--306, 2000.

\bibitem{khammash2022cybergenetics}
M.~H. Khammash, ``Cybergenetics: Theory and applications of genetic control systems,'' \emph{Proceedings of the IEEE}, vol. 110, no.~5, pp. 631--658, 2022.

\bibitem{naasell2003extension}
I.~N{\aa}sell, ``An extension of the moment closure method,'' \emph{Theoretical Population Biology}, vol.~64, no.~2, pp. 233--239, 2003.

\bibitem{hespanha2008moment}
J.~Hespanha, ``Moment closure for biochemical networks,'' in \emph{Proceedings of 2008 3rd International Symposium on Communications, Control and Signal Processing}.\hskip 1em plus 0.5em minus 0.4em\relax IEEE, 2008, pp. 142--147.

\bibitem{singh2010approximate}
A.~Singh and J.~P. Hespanha, ``Approximate moment dynamics for chemically reacting systems,'' \emph{IEEE Transactions on Automatic Control}, vol.~56, no.~2, pp. 414--418, 2010.

\bibitem{sakurai2017convex}
Y.~Sakurai and Y.~Hori, ``A convex approach to steady state moment analysis for stochastic chemical reactions,'' in \emph{Proceedings of 2017 IEEE 56th Annual Conference on Decision and Control (CDC)}.\hskip 1em plus 0.5em minus 0.4em\relax IEEE, 2017, pp. 1206--1211.

\bibitem{sakurai2018optimization}
------, ``Optimization-based synthesis of stochastic biocircuits with statistical specifications,'' \emph{Journal of The Royal Society Interface}, vol.~15, no. 138, p. 20170709, 2018.

\bibitem{sakurai2022interval}
------, ``Interval analysis of worst-case stationary moments for stochastic chemical reactions with uncertain parameters,'' \emph{Automatica}, vol. 146, p. 110647, 2022.

\bibitem{dowdy2018bounds}
G.~R. Dowdy and P.~I. Barton, ``Bounds on stochastic chemical kinetic systems at steady state,'' \emph{The Journal of Chemical Physics}, vol. 148, no.~8, p. 084106, 2018.

\bibitem{ghusinga2017exact}
K.~R. Ghusinga, C.~A. Vargas-Garcia, A.~Lamperski, and A.~Singh, ``Exact lower and upper bounds on stationary moments in stochastic biochemical systems,'' \emph{Physical Biology}, vol.~14, no.~4, p. 04LT01, 2017.

\bibitem{kuntz2019bounding}
J.~Kuntz, P.~Thomas, G.-B. Stan, and M.~Barahona, ``Bounding the stationary distributions of the chemical master equation via mathematical programming,'' \emph{The Journal of chemical physics}, vol. 151, no.~3, p. 034109, 2019.

\bibitem{li2025moment}
Z.~Li, M.~Barahona, and P.~Thomas, ``Moment-based parameter inference with error guarantees for stochastic reaction networks,'' \emph{The Journal of Chemical Physics}, vol. 162, no.~13, 2025.

\bibitem{sakurai2018bounding}
Y.~Sakurai and Y.~Hori, ``Bounding transient moments of stochastic chemical reactions,'' \emph{IEEE Control Systems Letters}, vol.~3, no.~2, pp. 290--295, 2018.

\bibitem{dowdy2018dynamic}
G.~R. Dowdy and P.~I. Barton, ``Dynamic bounds on stochastic chemical kinetic systems using semidefinite programming,'' \emph{The Journal of Chemical Physics}, vol. 149, no.~7, p. 074103, 2018.

\bibitem{holtorf2024tighter}
F.~Holtorf and P.~I. Barton, ``Tighter bounds on transient moments of stochastic chemical systems,'' \emph{Journal of Optimization Theory and Applications}, vol. 200, no.~1, pp. 104--149, 2024.

\bibitem{fukuda2001exploiting}
M.~Fukuda, M.~Kojima, K.~Murota, and K.~Nakata, ``Exploiting sparsity in semidefinite programming via matrix completion i: General framework,'' \emph{SIAM Journal on optimization}, vol.~11, no.~3, pp. 647--674, 2001.

\bibitem{wang2021chordal}
J.~Wang, V.~Magron, and J.-B. Lasserre, ``Chordal-{TSSOS}: a moment-sos hierarchy that exploits term sparsity with chordal extension,'' \emph{SIAM Journal on optimization}, vol.~31, no.~1, pp. 114--141, 2021.

\bibitem{zheng2021chordal}
Y.~Zheng, G.~Fantuzzi, and A.~Papachristodoulou, ``Chordal and factor-width decompositions for scalable semidefinite and polynomial optimization,'' \emph{Annual Reviews in Control}, vol.~52, pp. 243--279, 2021.

\bibitem{klep2025exploiting}
I.~Klep, V.~Magron, T.~Metzlaff, and J.~Wang, ``Exploiting term sparsity in symmetry-adapted basis for polynomial optimization,'' \emph{arXiv preprint arXiv:2511.18019}, 2025.

\bibitem{nesterov2013introductory}
Y.~Nesterov, \emph{Introductory lectures on convex optimization: A basic course}.\hskip 1em plus 0.5em minus 0.4em\relax Springer Science \& Business Media, 2013, vol.~87.

\bibitem{wan2020synthetic}
X.~Wan, F.~Pinto, L.~Yu, and B.~Wang, ``Synthetic protein-binding dna sponge as a tool to tune gene expression and mitigate protein toxicity,'' \emph{Nature communications}, vol.~11, no.~1, p. 5961, 2020.

\bibitem{goulart2024clarabel}
P.~J. Goulart and Y.~Chen, ``Clarabel: An interior-point solver for conic programs with quadratic objectives,'' \emph{arXiv preprint arXiv:2405.12762}, 2024.

\bibitem{garstka2020clique}
M.~Garstka, M.~Cannon, and P.~Goulart, ``A clique graph based merging strategy for decomposable {SDPs},'' \emph{IFAC-PapersOnLine}, vol.~53, no.~2, pp. 7355--7361, 2020.

\bibitem{grone1984positive}
R.~Grone, C.~R. Johnson, E.~M. S{\'a}, and H.~Wolkowicz, ``Positive definite completions of partial hermitian matrices,'' \emph{Linear algebra and its applications}, vol.~58, pp. 109--124, 1984.

\end{thebibliography}

\appendix
\subsection{Definition of $\Phi_i$}\label{app:def_Phi_i}
Fix a reaction index $i$ and consider the propensity function
$w_i(\bm{x})$. For zero-th order, unimolecular, and
homo-bimolecular reactions (see Table \ref{propent}),
we define the multi-index $\bm{d}$ as the exponent of the
highest-degree monomial in $w_i(\bm{x})$.
For hetero-bimolecular reactions, \textit{i.e.},
$w_i(\bm{x}) = \theta_i x_j x_k$ with $j<k$,
we define $\bm{d} := \bm{e}_j$, where $\bm{e}_j$
denotes the $j$-th standard basis vector. 
The set $\Phi_i$ is defined by 
\begin{align}
    \Phi_{i} := \left\{ \bm{x}^{\bm{\beta}} \middle|
   \begin{array}{l}
    \bm{\beta} \in \mathbb{N}_0^n, \\
         0\leq |\bm{\beta}| \leq \lceil(\mu-2)/2\rceil+\lceil |\bm{d}|/2 \rceil \\
         \lceil d_j/2\rceil\leq \beta_j \leq\lceil(\mu-2)/2\rceil+\lceil d_j/2 \rceil\\ \quad  \quad \quad \quad \quad\quad \quad \quad(j=1,2,\cdots,n)\\
   \end{array}
   \right\}.\notag
\end{align}
\subsection{Definition of $\Omega_i^{(k)}$}\label{app:def_Omega_i}
Let $\mathcal{D}_{i}^{(k)}$ be the set of monomials obtained by dividing the monomials appearing in the propensity function $w_i(\bm{x})$ by $x_k$, corresponding to the localizing matrix $N^{(k)}$.
For each monomial $\bm{x}^{\bm{\kappa}} \in \mathcal{D}_{i}^{(k)}$ with multi-index $\bm{\kappa} = [\kappa_1,\kappa_2, \ldots, \kappa_n]^{\top} \in \mathbb{N}_0^n$, we define the subset $\Lambda_{\bm{\kappa}}$ to construct the monomial set $\Omega_{i}^{(k)}$ as follows:
\begin{align}
    &\Lambda_{\bm{\kappa}} := \left\{ \bm{x}^{\bm{\beta}} ~\middle|~
   \begin{array}{l}
         \bm{\beta} \in \mathbb{N}_0^n, \\
         0\leq |\bm{\beta}| \leq \lceil(\mu+|\bm{\kappa}|-2)/2\rceil, \\
         \lceil \kappa_j/2\rceil\leq \beta_j \leq\lceil(\mu+\kappa_j-2)/2\rceil \\
         \hspace{60pt}(j=1,2,\dots,n)
   \end{array}
   \right\}, \notag \\
   &\Omega_{i}^{(k)} = \bigcup_{\bm{x}^{\bm{\kappa}}\in\mathcal{D}_{i}^{(k)}} \Lambda_{\bm{\kappa}}. \notag
\end{align}

\subsection{Rate constants used in the simulation}
\label{sec:rateconst}

The following values were used in the simulations for the analysis of Fig.~\ref{fig:7var_RW_bound_time}. These parameter values were chosen to reflect intracellular scales in \textit{E. coli}\cite{del2015biomolecular}. : $\theta_1 = \theta_7 = 0.2~\text{min}^{-1}$, $\theta_2 = \theta_8 = \ln (2)~\text{min}^{-1}$, $\theta_3 = \theta_9 = 0.5~\text{min}^{-1}$, $\theta_4 = \theta_{10} = \ln(2)/10~\text{min}^{-1}$, $\theta_{5} = \theta_{13} = 1.0~\text{copy}^{-1}\text{min}^{-1}$, $\theta_6 = \theta_{14} = 6.0~\text{min}^{-1}$, $\theta_{11} = 1.0~\text{min}^{-1}$, $\theta_{12} = 6.0~\text{copy}^{-1}\text{min}^{-1}$, $D_{T_1} = 50~\text{copy}$, $D_{T_2} = S_{\text{tot}} = 10~\text{copy}$.

\end{document}